\documentclass{article}

\usepackage{amsmath,amssymb}
\usepackage[mathscr]{eucal}
\usepackage{tikz}
\usetikzlibrary{positioning,automata}
\usetikzlibrary{arrows,calc}
\usetikzlibrary{matrix,shapes}

\usepackage{multicol}

\newcommand{\Pa}{\Pi}
\newcommand{\pcset}{\mathscr{P}}

\usepackage{authblk}
\usepackage{enumitem}
\usepackage{amsthm}
\usepackage{thmtools}

\newcommand{\lsub}[2]{{#1}_{#2}}

\theoremstyle{plain}
\newtheorem{theorem}{Theorem}[section]
\newtheorem{lemma}[theorem]{Lemma}
\newtheorem{proposition}[theorem]{Proposition}
\newtheorem{corollary}[theorem]{Corollary}

\theoremstyle{definition}

\newtheorem{question}[theorem]{Question}

\theoremstyle{remark}

\makeatletter
\renewcommand{\fnum@figure}{Fig. \thefigure}
\makeatother

\begin{document}

\title{On Orbits and the Finiteness of\\Bounded Automaton Groups}

\author{Ievgen Bondarenko}

\affil{\small
  Mechanics and Mathematics Faculty,\\
  Taras Shevchenko National University of Kyiv,\\ vul.Volodymyrska 64, 01033, Kyiv, Ukraine\\
  ievgen.bondarenko@gmail.com}

\author{Jan Philipp W\"achter}

\affil{\small
  Institut f\"ur formale Methoden der Informatik,\\
  Universit\"at Stuttgart,\\
  Universit\"atsstra{\ss}e 38, 70569 Stuttgart, Germany\\
  jan-philipp.waechter@fmi.uni-stuttgart.de}

\maketitle

\begin{abstract}
  We devise an algorithm which, given a bounded automaton $A$, decides whether the group generated by $A$ is finite. The solution comes from a description of the infinite sequences having an infinite $A$-orbit using a deterministic finite-state acceptor. This acceptor can also be used to decide whether the bounded automaton acts level-transitively.
  
  \noindent\footnotesize\emph{Keywords: automaton group; bounded automaton; finiteness problem; orbit.}\\
  \emph{Mathematics Subject Classification 2010: 20F10, 20M35}
\end{abstract}

\section{Introduction}\enlargethispage{\baselineskip}

One of many beautiful connections between combinatorics and group theory is the theory of automaton groups. An invertible Mealy automaton $A$ (further, just automaton) over an input-output alphabet $X$ generates a transformation group $G_A$ on the space of words. The recurrence nature of the automaton action produces complicated combinatorics on words. This is reflected in many interesting properties of automaton groups. It was discovered that even simple automata generate Burnside groups, groups of intermediate growth, just-infinite groups etc. (see \cite{GNS,self_sim_groups}).

Algorithmic problems around automaton groups are natural to study. Every automaton group has a decidable word problem. In fact, the word problem is PSPACE-complete \cite{WW:pspace}. Automaton groups with undecidable conjugacy problem were constructed in \cite{SV:conjugacy} and with undecidable order problem in \cite{G:Order,BM:WordOrder}. The isomorphism problem is undecidable in this class of groups (this follows from \cite{SV:conjugacy}).

Algorithmic problems of another type --- decide, given an automaton $A$, whether the group $G_A$ possesses a certain property --- remain widely open. In particular, the finiteness problem (decide whether the group $G_A$ is finite) is open. Interestingly, the finiteness problem for groups generated by bireversible automata is equivalent to the open irreducibility problem for complete square complexes and for lattices in the product of two trees (see \cite{BK:Square}). In \cite{G:Finite}, it is proved that the finiteness problem is undecidable for automaton semigroups.

In this paper we solve the finiteness problem for groups generated by bounded automata. An automaton $A$ is bounded if there is a constant $C$ such that the computation $A(w)$ ends at a non-trivial state for at most $C$ input strings $w$ of length $n$ for every $n$.
Bounded automata were introduced in \cite{sidki:circ} in order to relate the cyclic structure of automata with algebraic properties of the automaton groups. The class of bounded automaton groups is sufficiently large to include most of the fundamental examples of automaton groups, and rather restrictive to be amenable to investigation. In particular, the conjugacy and order problems are decidable in the group of all bounded automata \cite{BSZ:conjugacy}.

\begin{theorem}\label{thm_bounded_finite}
There exists an algorithm which, given a bounded automaton $A$, decides whether the automaton group $G_A$ is finite.
\end{theorem}

Our approach is constructive and relies on the following simple observation: Since every automaton group is finitely generated, the group $G_A$ is finite if and only if the orbits of the action $(G_A,X^{*})$ are uniformly bounded. We show not only how to recognize this property for bounded automata, but also give a description of the infinite sequences $w\in X^{\omega}$ having finite/infinite orbits:

\begin{theorem}\label{thm:orbit}
Let $A$ be a bounded automaton. There exists a constructible deterministic finite-state acceptor $R$ with the following property: A sequence $w\in X^{\omega}$ has a finite $G_A$-orbit if and only if the machine $R$, while processing the sequence $w$, passes only through a finite number of accepting states.\footnote[1]{For readers familiar with $\omega$-languages the theorem can be formulated as follows: The set of sequences $w\in X^{\omega}$ with a finite/infinite $G_A$-orbit is deterministic $\omega$-regular, i.e., it is accepted by a deterministic B\"{u}chi automaton $R$ computable from $A$.}
\end{theorem}

\begin{corollary}
There exists an algorithm which, given a bounded automaton $A$ and an eventually periodic sequence $w\in X^{\omega}$, decides whether the $G_A$-orbit of $w$ is finite.
\end{corollary}

Another decision problem --- transitivity on the levels --- is open not only for automaton groups but also for automaton transformations (see \cite{S:spherical} for partial solution of the later case). In the case of a binary alphabet, an automaton group is level-transitive if and only if it is infinite (see \cite[Proposition~2]{BGK:threeState}). Using the machine $R$ from Theorem~\ref{thm:orbit} we solve the level-transitivity problem for bounded automata:

\begin{corollary}
There exists an algorithm which, given a bounded automaton $A$, decides whether the group $G_A$ is level-transitive.
\end{corollary}

\section{Preliminaries: automata, actions, groups}

In this section we give a brief introduction into automaton groups, see \cite{GNS,self_sim_groups} for more details.

\paragraph{Spaces of words.} Let $X$ be a finite alphabet with at least two letters, fixed for the whole paper. Let $X^{*}$ be the space of all finite words over $X$ (including the empty word). The length of a word $v\in X^{*}$ is denoted by $|v|$. The space $X^{*}$ can be represented as the union of the levels $X^n$ for $n\geq 0$, where $X^n$ denotes the set of words of length $n$. We also consider the spaces of left- and right-infinite words denoted by $X^{-\omega}$ and $X^{\omega}$ respectively. For $w \in X^\omega$, we write $w_n$ for the finite prefix of length $n$ of $w$ and, symmetrically, for $w \in X^{- \omega}$, we write $\lsub{w}{n}$ for the suffix of length $n$.

\paragraph{Automata.} A finite automaton $A$ over $X$ is a tuple $A=(S,X,\pi:S\times X\rightarrow X\times S)$, where $S$ is a finite set, called the set of states of $A$, and $\pi$ is its transition-output function. We identify $A$ with its standard graph representation ---
a directed labeled graph, whose vertices are the states of $A$ and arrows are given by the rule:
\[
\mbox{ $s\xrightarrow{x|y}t$ \ whenever \ $\pi(s,x)=(y,t)$. }
\]
Informally, this arrow indicates that if the automaton is at state $s$ and reads letter $x$, then it outputs letter $y$ and changes its active state to $t$.

A subautomaton of $A$ is an automaton $B = (T, X, \rho: T \times Y \rightarrow X \times T)$ over $X$ with $T \subset S$ and $\rho(t, x) = \pi(t, x)$ for all $t \in T$ and $x \in X$ (i.e., we only consider subautomata over the same alphabet).

\paragraph{Automaton action.} By fixing an initial state $s$ of an automaton $A$, we get a transformation on the space $X^{*}$ in a standard way: Given an input word $v=x_1x_2\ldots x_n\in X^{*}$, the automaton processes the word letter by letter, and generates an output word $y_1y_2\ldots y_n$. In this way we get an action of states on the space $X^{*}$; explicitly, we have $s(x_1x_2\ldots x_n)=y_1y_2\ldots y_n$ whenever there is a directed path in $A$ of the form
\[
s=s_1\xrightarrow{x_1|y_1}s_2\xrightarrow{x_2|y_2}s_3\xrightarrow{x_3|y_3}\ldots\xrightarrow{x_n|y_n}s_{n+1}=t.
\]
The final state $t$ is called the section of $s$ at $v=x_1x_2\ldots x_n$ and is denoted by $s|_v$.
This action naturally extends to the action on the space $X^{\omega}$.

Since we are interested in automaton transformations rather than in automata themselves, all automata in this paper are assumed to be minimal, that is, different states define different transformations. Hence, we can identify the states of automata with the corresponding transformations. The trivial state $e$ is a state corresponding to the identity transformation. Algorithmically, we can minimize any automaton using standard techniques (see \cite{hopcroft1979automata}).

\paragraph{Automaton groups.} An automaton $A$ is called invertible, if every state of $A$ defines an invertible transformation on the space of words. This is the case iff the action of every state on $X$ is by permutation, which can easily be checked algorithmically. The inverse transformations are given by the inverse automaton $A^{-1}$. For an invertible $A$, the group generated by the transformations under composition is called the automaton group $G_A$ associated to $A$ (or generated by $A$). Every automaton group $G_A$ acts faithfully on the spaces $X^{*}$ and $X^{\omega}$ preserving the length of words and the length of common prefixes.

\paragraph{Sections.} Let $g$ be an automaton transformation given by a state $s$ of an automaton $A$. A section of $g$ at $v\in X^{*}$ is the transformation denoted $g|_v$ given by the final state of $A$ on input $v$.
Explicitly, $g|_v$ is the transformation defined by the rule: $g|_v(w_1)=w_2$ iff $g(vw_1)=g(v)w_2$. Sections have the following properties:
\[
g|_{vu}=g|_v|_u \ \mbox{ and } \ (gh)|_v=g|_{h(v)}h|_v.
\]
Every automaton group $G_A$ satisfies the self-similarity property: $g|_v\in G_A$ for every $g\in G_A$ and $v\in X^{*}$.

\begin{lemma}\label{lemma_finite_cyclic_part}
Let $A$ be a finite invertible automaton and $B\subset A$ be a subautomaton. Suppose there exists an integer $n$ such that $s|_v\in B$ for every $s\in A$ and $v\in X^n$. Then $G_A$ is finite if and only if $G_B$ is finite.
\end{lemma}
\begin{proof}
The condition of the lemma implies that $g|_v\in G_B$ for every $g\in G_A$ and $v\in X^n$. Since every $g$ is uniquely given by its action on $X^n$ and the tuple $(g|_v)_{v\in X^n}$, the claim follows.
\end{proof}

Let $C(A)$ be the subautomaton of $A$ spanned by the states reachable from cycles (the circuit part of $A$). Then $G_A$ is finite if and only if $G_{C(A)}$ is finite.

\paragraph{Post-critical set.} If $\ldots, x_2|y_2, x_1|y_1$ is the sequence of labels of a directed left-infinite path in $A$ ending at a non-trivial state, then we call the sequences $\ldots x_2x_1$ and $\ldots y_2y_1$ post-critical. The set of all post-critical sequences is called the post-critical set $\pcset$ of $A$.\footnote{The term post-critical comes from the connection between automaton groups and post-critically finite self-similar sets (see \cite{self_sim_groups}).} The set $\pcset$ is invariant under the shift: if $p=qx\in\pcset$, then $q\in\pcset$. We say that a word $v\in X^{n}$ is post-critical if $v = \lsub{p}{n}$ for some $p\in\pcset$. Note that $v$ is post-critical if $s|_v\neq e$ for some state $s\in C(A^{\pm 1})$ (see Fig.~\ref{fig:postCrit}).
\begin{figure}\centering
  \begin{tikzpicture}[auto, shorten >=1pt, >=latex, initial text=, node distance=1.5cm]
    \node[state] (t) {$t$};
    \node[right=of t] (ldots) {$\dots$};
    \node[state, right=of ldots] (s) {$s$};
    \node[right=of s] (rdots) {$\dots$};
    \node[state, right=of rdots] (sv) {$s|_v$};
    \node[above=0.75cm of t] (tdots) {$\dots$};

    \path[->] (t) edge node {$y_m/y_m'$} (ldots)
                  edge[bend right] node[above right] {$z_l/z_l'$} (tdots.south east)
              (ldots) edge node {$y_1/y_1'$} (s)
              (s) edge node {$x_n/x_n'$} (rdots)
              (rdots) edge node {$x_1/x_1'$} (sv)
              (tdots.south west) edge[bend right] node[above left] {$z_1/z_1'$} (t)
    ;
  \end{tikzpicture}
  \caption{If $s$ is reachable from a cycle and $s|_v \neq e$, then $v = x_n \dots x_1$ is post-critical as a suffix of the post-critical sequence $(z_l \dots z_1)^{-\omega} y_m \dots y_1 \, x_n \dots x_1$.}\label{fig:postCrit}
\end{figure}
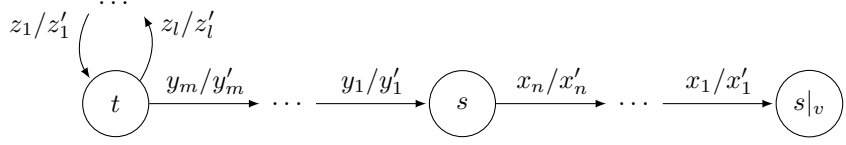
Post-critical sequences indicate directions with non-trivial actions.

\paragraph{Bounded automata.} An invertible automaton $A$ is called bounded if its post-critical set $\pcset$ is finite. This property can be characterized by the cyclic structure of the automaton: an invertible automaton $A$ is bounded if and only if different directed cycles in $A\setminus\{e\}$ are disjoint and not connected by a directed path. In particular, the boundedness of automata is a decidable property. Another characterization can be given in terms of sections: an invertible automaton $A$ is bounded if and only if there exists a constant $C$ such that
\[
  |\{v\in X^{n}\,:\, s|_v\neq e\}| \leq C \ \mbox{ for all $n\in\mathbb{N}$ and $s\in A$.}
\]
The latter characterization is Sidki's original definition for bounded automata and equivalent to the characterization using the cyclic structure \cite[Corollary~14]{sidki:circ}. The set of all transformations defined by all bounded automata forms a group called the group of bounded automata.

Every post-critical sequence of a bounded automaton is eventually periodic, periodic post-critical sequences are read along cycles in $A$.

\section{Action of bounded automata}

Let $A$ be a bounded automaton, $\pcset$ be its finite post-critical set, and $G$ the associated automaton group. In view of Lemma~\ref{lemma_finite_cyclic_part}, we may assume $A=C(A)$. Furthermore, we assume $\pcset \neq \emptyset$ as the group is trivial otherwise.

\paragraph{Orbits and $e$-orbits.} Let $O(v)$ denote the orbit of a word $v\in X^{*}$ under the $G$-action. As mentioned in the introduction, the automaton group $G$ is finite if and only if there is a constant bounding the size of all orbits $O(v)$ with $v\in X^{*}$. We can try to trace how the size of orbits change from level to level. However, there is one difficulty: the words $vx$ and $ux$ could be in different orbits, even when the words $v$ and $u$ are in the same orbit. To overcome this problem we subdivide orbits into parts called $e$-orbits.

We say that two words $v,u\in X^n$ are $e$-equivalent if there exists $s\in A^{\pm 1}$ such that $s(v)=u$ and $s|_v=e$. This generates an equivalence relation on $X^n$, the equivalence classes are called $e$-orbits of the $n$-th level. More explicitly, two words $v,u\in X^n$ belong to the same $e$-orbit whenever there exists a finite sequence of states $s_1,\ldots,s_m$ in $A^{\pm 1}$ such that
\[
v_1=v, v_2=s_1(v_1), v_3=s_2(v_2),\ldots, v_{m+1}=s_m(v_m)=u \ \mbox{ and } \ s_i|_{v_i}=e \ \mbox{ for all $i$} \text{.}
\]
Note that $g(v)=u$ and $g|_v=e$ for $g=s_m\ldots s_2s_1\in G$.\footnote{Therefore our equivalence refines the slightly different one where $v$ and $u$ are equivalent if there is some $g \in G$ with $g(v) = u$ and $g|_v = e$. The difference is that in our case all the individual $s_i|_{v_i}$ must be $e$. This is necessary for the construction outlined below.}

The $e$-orbit of a word $v\in X^{*}$ is denoted by $O^e(v)$. By construction, $O^e(v)\subset O(v)$ and $O^e(v)u\subset O^e(vu)$ for every $v,u\in X^{*}$. If $O^e(v)$ does not contain a post-critical word, then $O^e(vu)=O^e(v)u$ for every $u\in X^{*}$. In particular, $|O^e(v)|=|O^e(wx)|$, where $wx$ with $x \in X$ is the largest prefix of $v$ such that $O^e(w)$ contains a post-critical word.

In general, $e$-orbits can be very different from standard orbits\footnote{It is worth mentioning that $e$-orbits appear naturally in connection with limit spaces of automaton groups. The orbits $O(v)$ are the connected components of the action graph $\Gamma(A,X^n)$, where $X^n$ is the vertex set and $v,u\in X^n$ are connected by an edge if $s(v)=u$ for some $s\in A^{\pm 1}$. The $e$-orbits are the connected components of tile graphs, which describe the intersection of tiles in the limit space. The iterative construction of $e$-orbits is just the inflation construction of tile graphs described in \cite{BDN:Ends}.}. However, since the automaton $A$ is bounded and $A=C(A)$, every orbit is a union of a bounded number of $e$-orbits:
\begin{lemma}\label{lem:unionOfEOrbits}
  $O(v)$ is the union of $O^e(v)$ and $O^e(u)$ for all post-critical words $u\in O(v)$.
\end{lemma}
\begin{proof}
  Clearly, all the mentioned $e$-orbits are contained in $O(v)$. For the other direction, let $w \in O(v)$. Thus, there are $s_1, \dots, s_m \in A^{\pm 1}$ with $s_m \dots s_1 (v) = w$. If we have $w \not\in O^e(v)$, there must be some $i$ with $s_i|_{v'} \neq e$ for $v' = s_{i - 1} \dots s_1 (v)$ and we may assume $i$ to be maximal with this property. We have that $w' = s_i (v') = s_i \dots s_1 (v)$ is post-critical and in $O(v)$. Additionally, we have $w \in O^e(w')$ because $w = s_m \dots s_{i + 1} (w')$.
\end{proof}\noindent
In particular, the group $G$ is finite if and only if the sizes of its $e$-orbits are bounded.

\paragraph{Iterative construction of $e$-orbits.} The $e$-orbits can be constructed iteratively as follows.
Let $O_1,\ldots,O_l$ be the $e$-orbits of the $n$-th level. Then every $e$-orbit of the $(n+1)$-st level is a union of certain sets $O_ix=\{vx : v\in O_i\}$ ($x\in X$), and we can tell exactly which of these sets should be combined. Two different sets $O_ix$ and $O_kz$ belong to the same $e$-orbit if and only if there exists a set $O_jy\neq O_ix$ such that $O_jy$ and $O_kz$ belong to the same $e$-orbit and there exist $s\in A^{\pm 1}$ and $v\in O_i$, $u\in O_j$ such that $s(vx)=uy$ and $s|_{vx}=e$. Then $s|_v\neq e$, since otherwise $x=y$ and $v,u$ belong to the same $e$-orbit $O_i=O_j$. Therefore, the words $v$ and $u$ are post-critical. Hence, we successively combine two sets containing $O_ix$ and $O_jy$ if there exist $p,q\in\pcset$ such that $\lsub{p}{n} \in O_i$, $\lsub{q}{n} \in O_j$ and the pair $\{(p,x),(q,y)\}$ belongs to the set
\[
E^e=\left\{ \{(p,x),(q,y)\} :
\begin{array}{l}
\mbox{ there is a directed left-infinite path in $A$} \\
\mbox{ ending at the trivial state and labeled by $px|qy$}
\end{array}
 \right\}.
\]
The resulting sets are exactly the $e$-orbits of the $(n+1)$-st level.

In particular, in order to trace the growth of $e$-orbits, we just need to know how the post-critical words are distributed among $e$-orbits. Since the post-critical words are parameterized by the post-critical set $\pcset$, we can talk about partitions of $\pcset$. More precisely, we say that two elements $p,q\in\pcset$ belong to the same $e$-orbit of the $n$-th level if $\lsub{p}{n}$ and $\lsub{q}{n}$ are in the same $e$-orbit.

\paragraph{The partitions of $\pcset$ induced by the $e$-orbits.}
Let $\Pa_n$ be the partition of the post-critical set $\pcset$ induced by the $e$-orbits of the $n$-th level. The sequence $\Pa_n$ can be constructed recursively as follows. We start with the partition $\Pa_0=\{\pcset\}$ of the $0$-th level. In order to go from level $n$ to level $n + 1$, we describe a construction to define a partition $\Pi'$ based on a partition $\Pi$ in such a way that, if $\Pi = \Pi_n$ is the partition at level $n$, then $\Pi'$ is the partition $\Pi_{n + 1} = \Pi'$ at level $n + 1$.

Consider the collection of sets $Px$ for $P\in\Pa$ and $x\in X$, note that these sets form a partition of $\pcset X\supset\pcset$. Combine two sets $P x$ and $Q y$ whenever there exists $p\in P$ and $q\in Q$ such that $\{(p,x),(q,y)\}\in E^e$. Iterating this process results in a partition $\Lambda=\{\Lambda_1,\ldots,\Lambda_k\}$ of $\pcset X$. This $\Lambda$ induces the partition $\Pa'$ of the set $\pcset$, which consists of those sets $\Lambda_i\cap\pcset$ that are nonempty:
\[
  \Pi' = \{ \Lambda_i \cap \pcset \mid \Lambda_i \cap \pcset \neq \emptyset \}
\]
By the iterative construction of the $e$-orbits, the set $\Pa'$ is exactly the partition of $\pcset$ induced by the $e$-orbits of the $(n+1)$-st level if $\Pa$ is the partition for level $n$.

\begin{lemma}
The partitions $\Pa_n$ form a sequence of refinements. In particular, the sequence $\Pa_n$ eventually stabilizes and involves at most $|\pcset|$ partitions.
\end{lemma}
\begin{proof}
The first step of the construction gives a refinement $\Pa_1$ of the trivial partition.
In order to show that $\Pa_{n + 1}$ is a refinement of $\Pa_n$, we show something more general: If $\Sigma_1$ and $\Sigma_2$ are two partitions of $\pcset$ and $\Sigma_2$ is a refinement of $\Sigma_1$, then $\Sigma_2'$ is also a refinement of $\Sigma_1'$. Two sets containing $P_2x$ and $Q_2y$ with $P_2, Q_2 \in \Sigma_2$ only get merged in the construction of $\Sigma_2'$ if there are $p' \in P_2$ and $q' \in Q_2$ with $\{ (p', x), (q', y) \} \in E^e$. Since we have $p, p' \in P_2$, we also have $p, p' \in P_1$ for some $P_1 \in \Sigma_1$ and, similarly, $q, q' \in Q_1$ for some $Q_1 \in \Sigma_1$. Thus, we also merge the sets containing $P_1x$ and $Q_1y$ in the construction of $\Sigma_1'$ from $\Sigma_1$.
\end{proof}

\paragraph{Bounded $e$-orbits.}
Let us construct a finite-state machine $R^e$ recognizing unbounded $e$-orbits. The machine $R^{e}$,
given an input word $v\in X^{*}$, will return the set $O^e(v) \cap \lsub{\pcset}{|v|}$, where $\lsub{\pcset}{n} = \{ \lsub{p}{n} \mid p\in\pcset \}$.\footnote{The construction is similar to the construction of the recognition automaton from \cite{BDN:Ends}, which, given an input word $v\in X^n$, returns the number of connected components in the punctured graph $\Gamma(A,X^{n})$ with vertex $v$ removed.}

First, we successively construct the partitions $\Pa_n$ until $\Pa'=\Pa$:
\[
\Pa_0=\{\pcset\}, \ \Pa_{1}=\Pa'_0, \ \Pa_{2}=\Pa'_1, \ \ldots, \ \Pa_{n_0+1}=\Pa'_{n_0}=\Pa_{n_0}.
\]
Now the states of $R^e$ are the pairs $(P,\Pa)$ for $P\in\Pa$, where $\Pa$ ranges among the constructed partitions. Also, we add a new state $\bot=(\emptyset,\{\pcset\})$, which will correspond to the case $O^e(v) \cap \lsub{\pcset}{|v|}=\emptyset$. The initial state is $(\pcset,\Pa_0)$.

We construct transition arrows as follows: for every $P\in\Pa$ and $x\in X$ we put an arrow $(P,\Pa)\xrightarrow{x} (P',\Pa')$ for $P'\in\Pa'$, if $Px\subset \Lambda_i$ and $P'=\Lambda_i\cap\pcset$ for some $i$ (see the definition of $\Lambda_i$ above). For every state $(P, \Pa)$ of $R^e$ and every $x\in X$, if the transition $(P, \Pa) \xrightarrow{x}$ is not yet defined (i.e., if $\Lambda_i \cap \pcset = \emptyset$), we put $(P, \Pa) \xrightarrow{x} \bot$.

We mark a state $(P', \Pa')$ as accepting if there are states $(P, \Pa)$ and $(Q, \Pa)$ as well as $x, y \in X$ with $P \neq Q$ or $x \neq y$ and arrows $(P, \Pa) \xrightarrow{x} (P', \Pa')$ and $(Q, \Pa) \xrightarrow{y} (P', \Pa')$. This means that, when passing from $\Pa$ to $\Pa'$, the construction of $P'$ involved taking a union of at least two parts $Px$ and $Qy$ from the previous level.

The following properties follow from construction. The arrow $(P,\Pi)\xrightarrow{x} (P',\Pi')$ indicates the following: if $\Pi_n=\Pi$, then $\Pi_{n+1}=\Pi'$ and $Ox\subset O'$, where $O$ and $O'$ are the $e$-orbits containing the sets $P_n$ and $P'_{n+1}$ respectively. In particular, for every input word $v\in X^{*}$, if $(P,\Pi)$ is the final state of $R^e$ after processing $v$, then $P=O^e(v)\cap \lsub{\pcset}{|v|}$ (equivalently, $O^e(v)=O^e(P)$). Moreover, the state $(P',\Pi')$ is accepting exactly when $|Ox|=|O|<|O'|$
(i.e., the size of the $e$-orbit increases). The next statement follows.

\begin{proposition}\label{prop:recogn_e_orbit}
For every $w\in X^{\omega}$, the size of the $e$-orbits $O^e(w_n)$ is bounded if and only if the machine $R^e$, while processing the sequence $w$, passes only through a finite number of accepting states.
\end{proposition}

We cannot apply the previous statement to post-critical sequences, because they are left-infinite. Instead, we give the following characterization.

\begin{proposition}\label{prop:unbounded_cycle}
For $p\in\pcset$, the size of the $e$-orbits $O^e(\lsub{p}{n})$ is unbounded if and only if the machine $R^e$ contains a state $(P,\Pi)$ such that $p\in P$ and $(P, \Pa)$ is reachable from a cycle with an accepting state.
\end{proposition}
\begin{proof}
Let $(P,\Pa)$ be a state belonging to a cycle. Then $\Pi=\Pi_{n}$ for all $n\geq n_0$. For $p\in P$, there is a word $v\in X^{l}$, where $l$ is the length of the cycle, such that $(\lsub{p}{n})v$ and $\lsub{p}{n+l}$ belong to the same $e$-orbit for all $n\geq n_0$.

If the cycle contains an accepting state, then $|O^e(\lsub{p}{n})| < |O^e(\lsub{p}{n}v)|=|O^e(\lsub{p}{n+l})|$, and the size of $O^e(\lsub{p}{n})$ is unbounded. Now if $q\in\pcset$ belongs to a part reachable from the part $P$, then there is a word $u\in X^{m}$ such that $(\lsub{p}{n})u$ and $\lsub{q}{n+m}$ are in the same $e$-orbit for all $n\geq n_0$. Therefore, the size of the $e$-orbits $O^e(\lsub{q}{n})$ is unbounded.

In the other direction, if the cycle does not contain an accepting state, then $|O^e(\lsub{p}{n})|=|O^e(\lsub{p}{n}v)|=|O^e(\lsub{p}{n+l})|$ for all $n\geq n_0$, and the size of $O^e(\lsub{p}{n})$ is bounded. Now for $Q\in\Pa$, consider all cycles in $R^e$ that possess a path to $(Q, \Pa)$. If none of them has an accepting state, the size of $O^e(\lsub{p}{n})$ is bounded by a constant $C$ for all $p$ belonging to a part from these cycles. Let $l$ be the largest distance between parts in $\Pa$. Then, for every $w\in O^e(\lsub{q}{n})$, we can write $w=uv$ with $|v|=k\leq l$ such that $u\in O^e(\lsub{p}{n-k})$, where $p\in P$ for some part $P$ from the cycles. Therefore, $O^e(\lsub{q}{n})$ is covered by the sets $O^e(\lsub{p}{n-k})v$ for $v\in X^{k}$, $k\leq l$, and $p$ belongs to a part from the cycles. Hence $|O^e(\lsub{q}{n})|\leq C|X|^{1+2+\ldots+l}$, and the size of $O^e(\lsub{q}{n})$ is bounded.
\end{proof}

\paragraph{The finiteness problem.} The following statement follows from the discussion above and completes the proof of Theorem~\ref{thm_bounded_finite}.

\begin{proposition}\label{prop:finiteIffNoCyle}
Let $A$ be a bounded automaton. The automaton group $G_A$ is finite if and only if the machine $R^e$ has no cycle with an accepting state.
\end{proposition}
\begin{proof}
By Proposition~\ref{prop:recogn_e_orbit}, the existence of such a cycle implies the existence of a sequence $w\in X^{\omega}$ such that the size of the $e$-orbits $O^e(w_n)$ is unbounded. Hence, the group $G_A$ is infinite.

Conversely, if none of the cycles contain an accepting state, then the $e$-orbits containing a post-critical word have bounded size by Proposition~\ref{prop:unbounded_cycle}. Then all $e$-orbits have bounded size too, and the group $G_A$ is finite.
\end{proof}

\paragraph{Orbits of the action $(G,X^{\omega})$.}
We can modify the previous construction and construct a finite-state machine $R$ recognizing those sequences $w\in X^{\omega}$ that have an infinite $G$-orbit. Unlike in the previous subsections, we cannot assume $A=C(A)$ in general, and we consider the case $A\neq C(A)$ separately. However, we still assume $\pcset \neq \emptyset$ as otherwise the group and, thus, all orbits are finite.

First, we handle the case $A=C(A)$. Recall that every orbit is a union of a bounded number of $e$-orbits (see Lemma~\ref{lem:unionOfEOrbits}). More explicitly, the orbit $O(v)$ for $v\in X^{n}$ can be constructed as follows. If $O(v)\cap \lsub{\pcset}{n} =\emptyset$, then $O(v)=O^e(v)$. In the case $O(v) \cap \lsub{\pcset}{n} \neq\emptyset$, we successively combine two $e$-orbits $O^e(\lsub{p}{n})$ and $O^e(\lsub{q}{n})$ for every $p,q\in\pcset$ such that
the pair $\{ p, q \}$ belongs to the set
\[
E=\left\{ \{p,q\} :
\begin{array}{l}
\mbox{ there is a directed left-infinite path in $A$ ending} \\
\mbox{ at a non-trivial state and labeled by $p|q$ with $p \neq q$}
\end{array}
 \right\}.
\]
The resulting set containing $O(v)\cap \lsub{\pcset}{n}$ is exactly $O(v)$.

The machine $R$ handles this iterative process and differs from $R^{e}$ essentially by accepting states. We relabel the states of the machine $R^e$ by the following rule. For each partition $\Pi$ among the states of $R^e$, we successively combine two sets $P,Q\in\Pi$ if there exist $p\in P$ and $q\in Q$ such that $\{p,q\}\in E$.
The resulting sets define a new partition $\Pi'$ of the post-critical set $\pcset$. Then we relabel the state $(P,\Pi)$ with the label $(P',\Pi')$, where $P'\in\Pi'$ is the unique part containing $P$. In this way we may get several distinct states with the same label. In the new machine a state $(P',\Pi')$ is accepting if there exists an accepting state $(P,\Pi)$ of $R^{e}$ such that $P\subset P'$.

It follows from construction that, for every input word $v \in X^{*}$, if $(P,\Pi)$ is the final state of $R$ after processing $v$, then $P=O(v)\cap \lsub{\pcset}{|v|}$. An accepting state indicates that, when we pass through it, at least one $e$-part of the orbit increases.
Then Proposition~\ref{prop:recogn_e_orbit} holds for the standard orbits and we get

\begin{proposition}\label{prop:RforGOrbits}
A sequence $w\in X^{\omega}$ has a finite $G$-orbit if and only if the machine $R$, while processing the sequence $w$, passes only through a finite number of accepting states.
\end{proposition}

An automaton group is level-transitive if for every level there is only one orbit (i.e., if $|O(v)| = |X|^{|v|}$ for all $v \in X^*$). This property can be determined from $R$ by construction.

\begin{proposition}
  The group $G_A$ (with $A = C(A)$) is level-transitive if and only if the state $\bot$ cannot be reached and all states in $R$ are labeled by $(\pcset, \{ \pcset \})$.
\end{proposition}

\paragraph{The case $A \neq C(A)$.}
Every $e$-orbit of the $(n+1)$-st level is a union of $e$-orbits of the $n$-th level. In order to trace which $e$-orbits should be merged, we have considered words $v$ such that $s(vu)=v'u'$ for some $u\neq u'$. In the case $A=C(A)$ such words $v$ are parameterized by the post-critical sequences; that is why we have considered partitions of the post-critical set by $e$-orbits. In the case $A \neq C(A)$, it is not sufficient to consider only the post-critical sequences; however, there is a similar finite parametrization (closed under removing the last letter). We need to keep track of words of the form $u v^i w$ where $v$ is the input or output label of a cycle in $A$ and $u$ and $w$ correspond to the part before and after the cycle on some directed path. Note that the lengths of $u$ and $w$ are uniformly bounded. Additionally, we also need to keep track of the prefixes and the labels of paths that enter $e$ without a cycle (whose length is again uniformly bounded). To parameterize this, we can take a finite set $\pcset'$ of languages that are either singleton sets $\{ u \}$ (for the short words) or of the form $u v^* w$. Then we just repeat the construction from the previous subsections but make some adaptions. We cannot consider partitions anymore; instead we use sets $\Sigma$ of disjoint subsets of $\pcset'$. An $S \in \Sigma$ should represent an $e$-orbit on some level and contain those languages in $\pcset'$ that have a nonempty intersection with this $e$-orbit. By adapting the merging strategy for $R^e$ and the relabeling for $R$ accordingly, we obtain a finite-state machine $R$ such that, for every input word $v \in X^{*}$, if $(S, \Sigma)$ is the final state of $R$ after processing $v$, then $S$ contains the languages from $\pcset'$ that have a nonempty intersection with $O(v)$.

With the proper modifications, we get Proposition~\ref{prop:finiteIffNoCyle} and Proposition~\ref{prop:RforGOrbits} in the case $A \neq C(A)$. The condition for level-transitivity is a bit different, however. The group acts level-transitively if and only if the fail state $\bot$ cannot be reached in $R$ and every state is labeled by $(S, \{ S \})$ for some $S \subset \pcset'$.

\paragraph{A few open problems.} The following problems remain open.

\begin{question}
Is the finiteness problem  decidable for finitely generated subgroups of the group of bounded automata?
\end{question}
\begin{question}
Is the torsion problem decidable for groups generated by bounded automata?
\end{question}

\clearpage
\section{Example}

In this section, we apply the construction outlined above to the example automaton with seven states over the alphabet $X=\{a,b,c,d\}$ shown in Fig.~\ref{fig:exampleAutomaton}.

\paragraph{Post-critical set.}
  The automaton admits the post-critical sequences:
  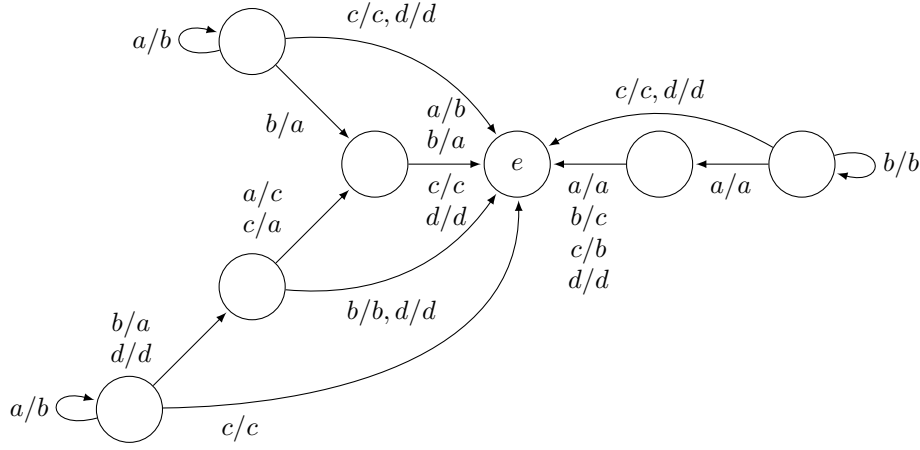
\begin{figure}\centering
    \begin{tikzpicture}[auto, shorten >=1pt, >=latex]
      \node[state] (e) {$e$};
      \node[state, left=of e] (1) {};
      \node[state, below left=of 1] (2) {};
      \node[state, below left=of 2] (3) {};
      \node[state, above left=of 1] (4) {};
      \node[state, right=of e] (5) {};
      \node[state, right=of 5] (6) {};

      \draw[->] (1) edge node[above, align=center] {$a/b$\\$b/a$} node[below, align=center] {$c/c$\\$d/d$} (e)
                (2) edge node[align=center, pos=0.2] {$a/c$\\$c/a$} (1)
                    edge[bend right] node[swap, pos=0.2] {$b/b, d/d$} (e)
                (3) edge[loop left] node {$a/b$} (3)
                    edge node[align=center, pos=0.1] {$b/a$\\$d/d$} (2)
                    edge[bend right, in=-120] node[swap, pos=0.1] {$c/c$} (e)
                (4) edge[loop left] node {$a/b$} (4)
                    edge node[swap] {$b/a$} (1)
                    edge[bend left] node[pos=0.2] {$c/c, d/d$} (e)
                (5) edge node[align=center] {$a/a$\\$b/c$\\$c/b$\\$d/d$} (e)
                (6) edge[loop right] node {$b/b$} (6)
                    edge node {$a/a$} (5)
                    edge[bend right] node[swap] {$c/c, d/d$} (e)
      ;
    \end{tikzpicture}
    \caption{Automaton used as an example. Self-loops at $e$ are not drawn.}
    \label{fig:exampleAutomaton}
  \end{figure}
  \begin{multicols}{4}
    \begin{enumerate}[itemsep=0pt]
      \item $a^{-\omega} b a$
      \item $a^{-\omega} b c$
      \item $a^{-\omega} d a$
      \item $a^{-\omega} d c$

      \columnbreak

      \item $a^{-\omega} b$
      \item $a^{-\omega} d$
      \item $a^{-\omega}$
      \vspace*{\fill}

      \columnbreak

      \item $b^{-\omega} a c$
      \item $b^{-\omega} a a$
      \item $b^{-\omega} d c$
      \item $b^{-\omega} d a$

      \columnbreak

      \item $b^{-\omega} a$
      \item $b^{-\omega} d$
      \item $b^{-\omega}$
      \vspace*{\fill}
    \end{enumerate}
  \end{multicols}
  \noindent{}We will use these numbers to refer to the post-critical sequences from now on.

\paragraph{Iterative construction of $e$-orbits.}
  For the construction of the partitions, we need the set $E^e$ arising from the left-infinite paths in the automaton. For our example, we have
  \allowdisplaybreaks\begin{align*}
    E^e = \{~%
    &\{ (1, a), (8, b) \}, &~& \{ (2, a), (9, b) \}, &~& \{ (5, a), (12, b) \}, &~& \{ (12, a), (12, a) \},\\
    &\{ (1, b), (8, a) \}, &~& \{ (2, b), (9, a) \}, &~& \{ (5, b), (12, a) \}, &~& \{ (12, b), (12, c) \},\\
    &  &~&  &~& \{ (5, b), (12, b) \}, &~& \\
    &\{ (1, c), (8, c) \}, &~& \{ (2, c), (9, c) \}, &~& \{ (5, c), (12, c) \}, &~& \{ (12, c), (12, b) \},\\
    &\{ (1, d), (8, d) \}, &~& \{ (2, d), (9, d) \}, &~& \{ (5, d), (12, d) \}, &~& \{ (12, d), (12, d) \},\\
    &\{ (6, b), (13, b) \}, &~& \{ (7, c), (14, c) \}, &~& \{ (14, c), (14, c) \}, &~& \\
    &\{ (6, d), (13, d) \}, &~& \{ (7, d), (14, d) \}, &~& \{ (14, d), (14, d) \}~\} \text{.} &~&
  \end{align*}

\paragraph{The partitions of $\pcset$ induced by the $e$-orbits.}
  Let $\Pa_n$ denote the partition of the post-critical sequences on the $n$-th level. We only describe one step in detail (going from $\Pa_1$ to $\Pa_2$) and only state the partitions for the other steps. We have $\Pa_1 = \{ P, Q \}$ where $Q$ contains the post-critical sequences 6 and 13 and $P$ contains all other words.

  Going to the next partition, we merge $Pa$ with $Pb$ because of $\{ (1, a), (8, b) \} \in E^e$ and $Pb$ with $Pc$ because of $\{ (12, b), (12, c) \} \in E^e$ but we do not merge any other sets.

  The post-critical sequences 3 and 11 belong to $Qa$, 4 and 10 to $Qc$ and 6 and 13 to $Pd$. All other post-critical sequences belong to $Pa \cup Pb \cup Pc$. This yields the step from $\Pi_1$ to $\Pi_2$ depicted in Fig.~\ref{fig:Pi1ToPi2}.
  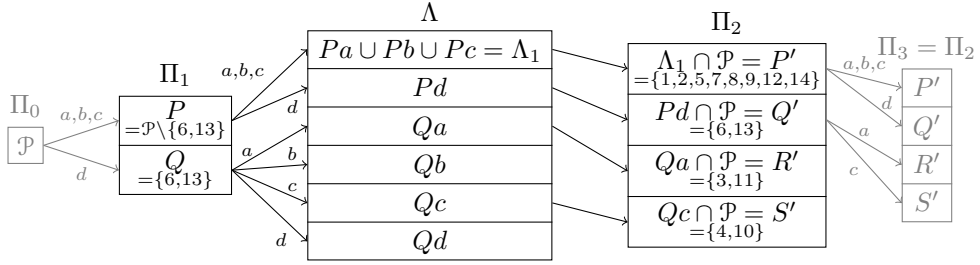
\begin{figure}\centering
    \begin{tikzpicture}[auto]
      \node[draw, gray] (pi0) {$\pcset$};
      \node[above=0cm of pi0, gray] {$\Pa_0$};

      \matrix[matrix of nodes, draw, right=of pi0, inner sep=0pt, nodes={align=center, inner sep=2pt}] (pi1) {
        \node[align=center] {$P$ \\[-0.5\baselineskip] $\scriptstyle = \pcset \setminus \{ 6, 13 \}$}; \\
        \node[align=center] {$Q$ \\[-0.5\baselineskip] $\scriptstyle = \{ 6, 13 \}$}; \\
      };
      \draw ($(pi1.north west)!1/2!(pi1.south west)$) -- ($(pi1.north east)!1/2!(pi1.south east)$);
      \node[above=0cm of pi1] {$\Pa_1$};

      \path[->, gray] (pi0.east) edge node[above] {$\scriptstyle a, b, c$} ($(pi1.north west)!1/4!(pi1.south west)$)
                                 edge node[below] {$\scriptstyle d$} ($(pi1.north west)!3/4!(pi1.south west)$)
      ;

      \matrix[matrix of math nodes, text height=1.5ex, text depth=0ex, draw, right=of pi1, inner sep=0pt, nodes={inner sep=4pt}] (Lambda) {
        Pa \cup Pb \cup Pc = \Lambda_1 \\
        Pd \\
        Qa \\
        Qb \\
        Qc \\
        Qd \\
      };
      \draw ($(Lambda.north west)!1/6!(Lambda.south west)$) -- ($(Lambda.north east)!1/6!(Lambda.south east)$);
      \draw ($(Lambda.north west)!2/6!(Lambda.south west)$) -- ($(Lambda.north east)!2/6!(Lambda.south east)$);
      \draw ($(Lambda.north west)!3/6!(Lambda.south west)$) -- ($(Lambda.north east)!3/6!(Lambda.south east)$);
      \draw ($(Lambda.north west)!4/6!(Lambda.south west)$) -- ($(Lambda.north east)!4/6!(Lambda.south east)$);
      \draw ($(Lambda.north west)!5/6!(Lambda.south west)$) -- ($(Lambda.north east)!5/6!(Lambda.south east)$);
      \node[above=0cm of Lambda] {$\Lambda$};

      \path[->] ($(pi1.north east)!1/4!(pi1.south east)$) edge node[inner sep=2pt] {$\scriptstyle a, b, c$} ($(Lambda.north west)!1/12!(Lambda.south west)$)
                edge node[below, pos=0.8, inner sep=2pt] {$\scriptstyle d$} ($(Lambda.north west)!3/12!(Lambda.south west)$)
                ($(pi1.north east)!3/4!(pi1.south east)$) edge node[above, pos=0.2, inner sep=2pt] {$\scriptstyle a$} ($(Lambda.north west)!5/12!(Lambda.south west)$)
                edge node[above, pos=0.8, inner sep=2pt] {$\scriptstyle b$} ($(Lambda.north west)!7/12!(Lambda.south west)$)
                edge node[above, pos=0.8, inner sep=2pt] {$\scriptstyle c$} ($(Lambda.north west)!9/12!(Lambda.south west)$)
                edge node[below left, pos=0.8, inner sep=2pt] {$\scriptstyle d$} ($(Lambda.north west)!11/12!(Lambda.south west)$)
      ;

      \matrix[matrix of nodes, draw, right=of Lambda, inner sep=0pt, nodes={align=center, inner sep=2pt}] (pi2) {
          \node[align=center] {$\Lambda_1 \cap \pcset = P'$ \\[-0.5\baselineskip] $\scriptstyle = \{ 1, 2, 5, 7, 8, 9, 12, 14 \}$}; \\
          \node[align=center] {$Pd \cap \pcset = Q'$ \\[-0.5\baselineskip] ${\scriptstyle = \{ 6, 13 \}}$}; \\
          \node[align=center] {$Qa \cap \pcset = R'$ \\[-0.5\baselineskip] ${\scriptstyle = \{ 3, 11 \}}$}; \\
          \node[align=center] {$Qc \cap \pcset = S'$ \\[-0.5\baselineskip] ${\scriptstyle = \{ 4, 10 \}}$}; \\
      };
      \node[above=0cm of pi2] {$\Pa_2$};
      \draw ($(pi2.north west)!1/4!(pi2.south west)$) -- ($(pi2.north east)!1/4!(pi2.south east)$);
      \draw ($(pi2.north west)!2/4!(pi2.south west)$) -- ($(pi2.north east)!2/4!(pi2.south east)$);
      \draw ($(pi2.north west)!3/4!(pi2.south west)$) -- ($(pi2.north east)!3/4!(pi2.south east)$);
      \path[->] ($(Lambda.north east)!1/12!(Lambda.south east)$) edge ($(pi2.north west)!1/8!(pi2.south west)$)
                ($(Lambda.north east)!3/12!(Lambda.south east)$) edge ($(pi2.north west)!3/8!(pi2.south west)$)
                ($(Lambda.north east)!5/12!(Lambda.south east)$) edge ($(pi2.north west)!5/8!(pi2.south west)$)
                ($(Lambda.north east)!9/12!(Lambda.south east)$) edge ($(pi2.north west)!7/8!(pi2.south west)$)
      ;

      \matrix[matrix of math nodes, gray, text height=1.5ex, text depth=0ex, inner sep=0pt, draw, nodes={inner sep=4pt}, right=of pi2] (pi3) {
        P' \\
        Q' \\
        R' \\
        S' \\
      };
      \node[above=0cm of pi3, gray] {$\Pa_3 = \Pa_2$};
      \draw[gray] ($(pi3.north west)!1/4!(pi3.south west)$) -- ($(pi3.north east)!1/4!(pi3.south east)$);
      \draw[gray] ($(pi3.north west)!2/4!(pi3.south west)$) -- ($(pi3.north east)!2/4!(pi3.south east)$);
      \draw[gray] ($(pi3.north west)!3/4!(pi3.south west)$) -- ($(pi3.north east)!3/4!(pi3.south east)$);

      \path[->, gray] ($(pi2.north east)!1/8!(pi2.south east)$) edge node[inner sep=2pt, above] {$\scriptstyle a, b, c$} ($(pi3.north west)!1/8!(pi3.south west)$)
                      edge node[inner sep=2pt, above, pos=0.8] {$\scriptstyle d$} ($(pi3.north west)!3/8!(pi3.south west)$)
                      ($(pi2.north east)!3/8!(pi2.south east)$) edge node[inner sep=2pt, above] {$\scriptstyle a$} ($(pi3.north west)!5/8!(pi3.south west)$)
                      edge node[inner sep=2pt, below left] {$\scriptstyle c$} ($(pi3.north west)!7/8!(pi3.south west)$)
      ;
    \end{tikzpicture}
    \caption{Construction of $\Pi_2$ from $\Pi_1$.}\label{fig:Pi1ToPi2}
  \end{figure}

\paragraph{Bounded $e$-orbits.}
  For $R^e$, we obtain the machine depicted in Fig.~\ref{fig:exampleRe}. The states $(P, \Pi_1)$ and $(P', \Pi_2)$ are accepting because they have in-going edges labeled with multiple different letters. Since the accepting state $(P', \Pi_2)$ is part of a cycle, the group generated by the automaton is infinite by Proposition~\ref{prop:finiteIffNoCyle}.
  \begin{figure}\centering
    \begin{tikzpicture}[auto, shorten >=1pt, >=latex, initial text=]
      \node[state, ellipse, initial] (0) {$\pcset, \{ \pcset \}$};
      \node[state, ellipse, above right=of 0, accepting] (1) {$P, \Pa_1$};
      \node[state, ellipse, below right=of 0] (2) {$Q, \Pa_1$};
      \node[state, ellipse, accepting, above right=0cm and 1.5cm of 1] (3) {$P', \Pa_2 \textcolor{gray}{'}$};
      \node[state, ellipse, below right=0.5cm and 1.5cm of 1] (4) {$Q', \Pa_2 \textcolor{gray}{'}$};
      \node[state, ellipse, below=0.5cm of 4] (5) {$R', \Pa_2$};
      \node[state, ellipse, below=0.5cm of 5] (6) {$S', \Pa_2$};

      \path[->] (0) edge node {$a, b, c$} (1)
                    edge node {$d$} (2)
                (1) edge node {$a, b, c$} (3)
                    edge node {$d$} (4)
                (2) edge node {$a$} (5)
                    edge node[swap] {$b$} (6)
                (3) edge[loop right] node {$a, b, c$} (3)
                    edge node {$d$} (4)
                (4) edge node {$a$} (5)
                    edge[bend left, out=80, in=100] node {$c$} (6)
      ;

      \node[above left=0cm of 5, gray] (5') {$R'', \Pa_2'$};
      \draw[gray] (5.south west) -- (5.north east);
      \node[below right=-0.25cm and 0.25cm of 6, gray] (6') {$R'', \Pa_2'$};
      \draw[gray] (6.south west) -- (6.north east);
    \end{tikzpicture}
    \caption{The machine $R^e$ (black part) and the machine $R$ (gray part) for the example automaton. All missing edges go to the state $\bot = (\emptyset, \{ \pcset \})$, which is not drawn.}\label{fig:exampleRe}
  \end{figure}
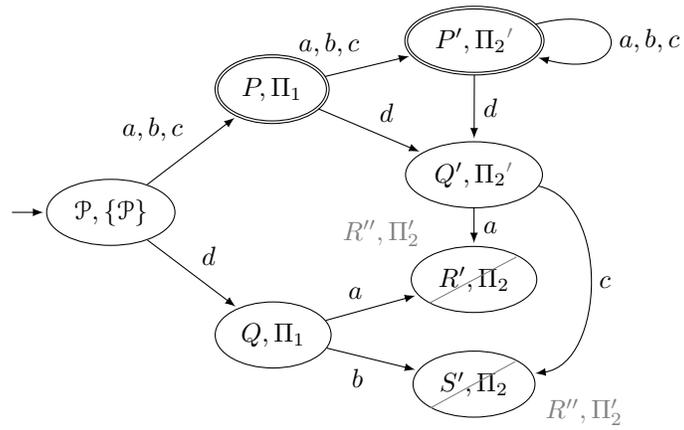~

\paragraph{Orbits of the action $(G,X^{\omega})$.}
  We have
  \[
    E = \{~\{1, 8\}, \{2, 9\}, \{3, 10\}, \{4, 11\}, \{5, 12\}, \{6, 13\}, \{7, 14\}~\} \text{.}
  \]
  In order to construct $R$ from $R^e$, we only have to modify $\Pi_2$ where we merge $R'$ and $S'$ into $R'' = \{ 3, 4, 10, 11 \}$ because we have $3 \in R'$, $10 \in S'$ and $\{3, 10\} \in E$. The pair $\{4, 11\}$ also requires us to merge $R'$ and $S'$ and all other pairs in $E$ already belong to the same set. In the labels, we replace all $\Pi_2$ by $\Pi_2'$ and $R'$ and $S'$ by $R''$. Thus, we now have two states labeled by $(R'', \Pi_2')$, which are not accepting because neither $(R', \Pi_2)$ nor $(S', \Pi_2)$ was. All other previously accepting states remain accepting.

  We see that the automaton does not act level-transitively. If we restrict the automaton to the alphabet $\{ a, b, c \}$, however, then the result does act level-transitively.


\clearpage
\begin{thebibliography}{20}

\bibitem{BM:WordOrder}
L. Bartholdi, I. Mitrofanov.
\newblock The word and order problems for self-similar and automata groups.
\newblock \textit{Groups, Geometry, and Dynamics} \textbf{14}(2) (2020) 705--728.

\bibitem{BSZ:conjugacy}
I. {Bondarenko}, N. {Bondarenko}, S. N. {Sidki}, and F. R. {Zapata}.
\newblock On the conjugacy problem for finite-state automorphisms of regular rooted trees. With an appendix by Rapha\"el M. Jungers.
\newblock \textit{Groups, Geometry, and Dynamics} \textbf{7}(2) (2013) 323--355.

\bibitem{BDN:Ends}
I.~Bondarenko, D.~D'Angeli, T.~Nagnibeda,
\newblock Ends of Schreier graphs and cut-points of limit spaces of self-similar groups,
\newblock \textit{Journal of Fractal Geometry}, Number 4, P. 369-424, 2017.

\bibitem{BGK:threeState}
I. {Bondarenko}, R.~I. {Grigorchuk}, R. Kravchenko, Y. Muntyan, V.~V. Nekrashevych, D. Savchuk, Z. {\v{S}uni\'c}.
\newblock Groups generated by 3-state automata over a 2-letter alphabet, I.
\newblock \textit{S\~ao Paulo Journal of Mathematical Sciences} \textbf{1}(1) (2007) 1--39.

\bibitem{BK:Square}
I.~Bondarenko, B. Kivva.
\newblock Automaton groups and complete square complexes.
\newblock Accepted in \textit{Groups, Geometry, and Dynamics}, arXiv:1707.00215.

\bibitem{G:Finite}
P. Gillibert,
\newblock The finiteness problem for automaton semigroups is undecidable,
\newblock \textit{International Journal of Algebra and Computation} \textbf{24}(1) (2014) 1--9.

\bibitem{G:Order}
P. Gillibert,
\newblock An automaton group with undecidable order and Engel problems,
\newblock \textit{Journal of Algebra} \textbf{497} (2018) 363--392.

\bibitem{GNS}
R.~I. Grigorchuk, V.~V. Nekrashevych, V.~I.~Sushchansky,
\newblock Automata, dynamical systems and groups,
\newblock \textit{Proceedings of the Steklov Institute of Mathematics} \textbf{231} (2000)
128--203.

\bibitem{hopcroft1979automata}
J.~E. Hopcroft and J.~D. Ullman.
\newblock {\em Introduction to Automata Theory, Languages, and Computation}.
\newblock Addison-Wesley, 1979.

\bibitem{self_sim_groups}
V. Nekrashevych. {\it Self-similar groups},
{Mathematical Surveys and Monographs}, Vol.117 (American Mathematical Society, Providence, 2005).

\bibitem{sidki:circ} S.~Sidki,
\newblock Automorphisms of one-rooted trees: growth, circuit structure, and acyclicity.
\newblock {\em Journal of Mathematical Sciences (New York).} \textbf{100} (2000), no.~1, 1925--1943.

\bibitem{S:spherical}
B. Steinberg,
\newblock Testing spherical transitivity in iterated wreath products of cyclic groups.
\newblock Preprint, 2006.

\bibitem{SV:conjugacy}
Z. {\v{S}uni\'c} and E. {Ventura},
\newblock The conjugacy problem in automaton groups is not solvable,
\newblock \textit{Journal of Algebra} \textbf{364} (2012) 148--154.

\bibitem{WW:pspace}
J.~Ph. W{\"{a}}chter and A. Wei{\ss}.
\newblock An automaton group with {PSPACE}-complete word problem.
\newblock In {\em 37th International Symposium on Theoretical Aspects of
  Computer Science, {STACS} 2020, March 10-13, 2020, Montpellier, France},
pages 6:1--6:17, 2020.


\end{thebibliography}
\end{document}